\documentclass[11pt,reqno]{amsart}
\usepackage{amssymb,latexsym}
\usepackage{graphicx}
\usepackage{url}		

\calclayout
\theoremstyle{plain}
\numberwithin{equation}{section}
\newtheorem{thm}{Theorem}[section]
\newtheorem{theorem}[thm]{Theorem}
\newtheorem{lemma}[thm]{Lemma}
\newtheorem{corollary}[thm]{Corollary}

\newtheorem{conjecture}[thm]{Conjecture}

\begin{document}
\title{Reciprocal Sum of Palindromes}
\author{Phakhinkon Phunphayap and Prapanpong Pongsriiam$^*$}
\address{Department of Mathematics, Faculty of Science\\
	Silpakorn University\\
	Nakhon Pathom\\
	73000, Thailand}
\email{phakhinkon@gmail.com}
\address{Department of Mathematics, Faculty of Science\\
                Silpakorn University\\
                Nakhon Pathom\\
               73000, Thailand}
\email{prapanpong@gmail.com, pongsriiam\_p@silpakorn.edu}
\thanks{*Prapanpong Pongsriiam receives financial support jointly from The Thailand Research Fund and Faculty of Science Silpakorn University, grant number RSA5980040. He is the corresponding author.}
\begin{abstract}
A positive integer $n$ is said to be a palindrome in base $b$ (or $b$-adic palindrome) if the representation of $n = (a_k a_{k-1} \cdots a_0)_b$ in base $b$ with $a_k \neq 0$ has the symmetric property $a_{k-i} = a_i$ for every $i=0,1,2,\ldots ,k$. Let $s_b$ be the reciprocal sum of all $b$-adic palindromes. It is not difficult to show that $s_b$ converges. In this article, we obtain upper and lower bounds for $s_b$ and the inequality $s_{b} <s_{b'}$ for $2\leq b<b'$. Its consequences and some numerical data are also given.

keywords: Palindrome, palindromic number, reciprocal, series, approximation 

\end{abstract}

\maketitle
\section{Introduction}

In recent years, there has been an increasing interest in the importance of palindromes in mathematics, theoretical computer science and theoretical physics. This stems from their role in the modeling of quasi-crystals in theoretical physics (see e.g. \cite{Dam,Hof}) and also Diophantine approximation in number theory \cite{Adam1,Adam2,Adam3,Adam4,Bug,Fis,Roy,Roy2}.

Throughout this article, $a,b,m,n,k,\ell$ are positive integers, $x,y$ are positive real numbers, $\left\lfloor x\right\rfloor$ is the largest integer less than or equal to $x$, $\left\lceil x\right\rceil$ is the smallest integer not less than $x$, and $\log x$ is the natural logarithm of $x$. Let $n = (a_ka_{k-1}\cdots a_0)_{10}$ be the decimal expansion of $n$ with $a_k \neq 0$. Then $n$ is said to be a palindrome if the digits of $n$ satisfy the symmetric property: $a_i = a_{k-i}$ for $0\leq i \leq \left\lfloor \frac{k}{2}\right\rfloor$. So, for example, $1, 2, 3, \ldots, 9$ are palindromes which have $1$ digit, $11, 22, 33, \ldots, 99$ are palindromes which have $2$ digits, and $101, 111, 121, \ldots, 191, 202, 212, \ldots, 292, \ldots, 909, 919, \ldots, 999$ are palindromes which have $3$ digits. The sequence of palindromes is the entry  A002113 in OEIS \cite{Slo}. See also the sequences A002385 and A002779 for palindromic primes and palindromic squares \cite{Slo}, respectively.
	
The definition of palindromic numbers can also be generalized to a general base. If $b>1$, then $n$ is said to be a palindrome in base $b$ (or a $b$-adic palindrome) if the $b$-adic expansion of $n=(a_k a_{k-1}\ldots a_0)_b$ with $a_k\neq 0$ satisfies $a_{k-i}=a_i$ for $0\leq i \leq \left\lfloor \frac{k}{2}\right\rfloor$. For convenience, if we write a number without specifying the base, then it is always in base 10. Furthermore, we let $P_b$ be the set of all palindromes in base $b$, $s_b$ the reciprocal sum of all $b$-adic palindromes, and $s_{b,k}$ the reciprocal sum of all $b$-adic palindromes which have $k$ digits in its $b$-adic expansion. The set $P_b$ is infinite but quite sparse, so it is not difficult to show that 
\begin{equation*}
	s_b=\sum_{\substack{n=1\\n\in P_b}}^\infty\frac{1}{n}\quad\text{converges}.
\end{equation*}
To verify the above assertion, we only need to recall monotone convergence theorem (see for example in \cite[Theorem 3.3.2]{BSh}) and observe that the number of $b$-adic palindromes which have $k$ digits is less than $b^{\frac{k+1}{2}}$. Therefore
\begin{equation*}
	s_{b,k} = \sum_{\substack{b^{k-1}\leq n < b^k\\ n\in P_b}} \frac{1}{n} \leq \frac{1}{b^{k-1}} \sum_{\substack{b^{k-1}\leq n < b^k\\ n\in P_b}} 1 \leq \frac{b^{\frac{k+1}{2}}}{b^{k-1}} = \frac{1}{b^{\frac{k-3}{2}}},
\end{equation*}
and thus 
\begin{equation*}
	s_b = \sum_{k=1}^\infty s_{b,k} \leq \sum_{k=1}^\infty \frac{1}{b^{\frac{k-3}{2}}} = b^{\frac{3}{2}}\sum_{k=1}^\infty \left( \frac{1}{\sqrt{b}} \right)^k = \frac{b^{\frac{3}{2}}}{\sqrt{b} - 1}.
\end{equation*}
It seems that there are some discussions on the reciprocal sum of $b$-adic palindromes on the internet (see, for instance, in \cite{Slo} and \cite{Math}) but as far as we are aware, our observation has not appeared in the literature. 

In this article, we obtain upper and lower bounds for each $s_b$, which enable us to show that $s_b <s_{b'}$ for $2\leq b<b'$. This inequality is not trivial since $s_{b+1}-s_b$ converges to zero as $b\rightarrow \infty$ (see Corollary \ref{rpcor4}). So a careful estimate is needed. We also show some numerical data and post an open question related to the reciprocal sum of palindromes at the end of this article.

Remark that the reciprocal sum of an integer sequence is also of general interest as proposed by Bayless and Klyve \cite{Bay} and by 
Roggero, Nardelli, and Di Noto \cite{Rog}. The additive property of palindromes has recently been investigated by Banks \cite{Bank}, Cilleruelo, Luca, and Baxter \cite{Cil}, and Rajasekaran, Shallit, and Smith \cite{Ra}. For other results concerning palindromes, we refer the reader to \cite{Adam1,Adam2,Adam3,Adam4,Amb,Bank2,Bank3,Brl,Bug,Dam,Fic,Fis,Goin,Grim,Hof,Luca,Roy,Roy2,Slo}.

\section{Main Results}

\begin{theorem}\label{rpthm1}
Let $b\geq 2$, $x_b = \sum_{a=1}^{b-1}\frac{1}{a}$, and $y_b = \sum_{a=2}^b\frac{1}{a}$. Then $s_b$ satisfies the following inequalities for all $\ell\geq 3$ and $m\geq 2$,
\begin{align*}
s_b &\geq \frac{b+2}{b+1}x_b+\frac{2y_b}{b^{\left\lceil\frac{\ell}{2}\right\rceil} - b^{\left\lceil\frac{\ell}{2}\right\rceil -1}} + \sum_{k=3}^{2\left\lceil \frac{\ell}{2}\right\rceil - 1} s_{b,k}, \\
s_b &\leq \left(\frac{b+2}{b+1}+\frac{2}{ b^{m}-b^{m-1}}+\sum_{k=\ell}^{2m-1} \frac{b^{\left\lceil \frac{k-2}{2}\right\rceil}}{b^{k-1}+1}\right)x_b + \sum_{k=3}^{\ell -1} s_{b,k}.
\end{align*}
\end{theorem}

As usual, if $2m-1<\ell$ or $\ell -1 <3$, the sum such as
\begin{equation*}
\sum_{k=\ell}^{2m-1} \frac{b^{\left\lceil \frac{k-2}{2}\right\rceil}}{b^{k-1}+1} \quad\text{or}\quad \sum_{k=3}^{\ell -1} s_{b,k}
\end{equation*}	
are considered to be zero. In addition, the upper and lower bounds given in Theorem \ref{rpthm1} may look complicated but it can lead us to the proof of Theorem \ref{rpthm2}.  Furthermore, we can apply Theorem \ref{rpthm1} with $m=2$ and $\ell=3$ to obtain a simpler bound as shown in the next corollary.

\begin{corollary}\label{rpcor1}
	Let $b\geq 2$, $x_b = \sum_{a=1}^{b-1}\frac{1}{a}$, and $y_b = \sum_{a=2}^b\frac{1}{a}$. Then
	\begin{equation*}
		\frac{b+2}{b+1}x_b+\frac{2y_b}{b(b-1)} \leq s_b \leq \left(\frac{b+2}{b+1}+\frac{2}{b(b-1)}+\frac{b}{b^2 +1}\right)x_b.
	\end{equation*}
\end{corollary} 

\begin{proof}
As discussed above this corollary follows from the substitution $m=2$ and $\ell = 3$ in Theorem \ref{rpthm1}.	
\end{proof}

In Corollary \ref{rpcor2}, we show that the difference between the upper and lower bounds for $s_b$ given in Corollary \ref{rpcor1} converges to zero as $b\rightarrow \infty$. This means that these bounds are good estimates for $s_b$. In fact, we can give an asymptotic formula for $s_b$ (see also in Corollary \ref{rpcor2}). Other results concerning the sequence $(s_b)_{b\geq 2}$ are also given in Corollaries \ref{rpcor3} and \ref{rpcor4}. Now we prove Theorem \ref{rpthm1}.

\begin{proof}[Proof of Theorem \ref{rpthm1}]
	For simplicity, we write $x$ and $y$ instead of $x_b$ and $y_b$, respectively. We first consider the reciprocal sum of palindromes in base $b$ which have $k$ digits in their $b$-adic expansion:
	\begin{equation*}
		s_{b,k} = \sum_{\substack{b^{k-1}\leq n < b^k\\ n\in P_b}}\frac{1}{n}.
	\end{equation*}
	Obviously when $k=1$, the sum is $1+\frac{1}{2}+\cdots +\frac{1}{b-1} = x$. For $k=2$, the sum is 
	\begin{equation}\label{thm1eq1}
		\frac{1}{(11)_b}+\frac{1}{(22)_b}+\cdots+\frac{1}{((b-1)(b-1))_b} = \frac{1}{b+1} + \frac{1}{2b+2} + \cdots + \frac{1}{(b-1)b+b-1} = \frac{x}{b+1}.
	\end{equation}
	Let $k\geq \ell$. The $b$-adic palindromes which have $k$ digits are of the form $(aa_1a_2\cdots a_{k-2}a)_{b}$ where $1\leq a\leq b-1$, $0\leq a_i \leq b-1$ for all $i\in \{1, 2, \ldots, k-2\}$, with the usual symmetry conditions on $a_i$. We fix $a$ and count the number of palindromes in this form. There are $b$ choices for $a_1\in \{0, 1, 2, \ldots, b-1\}$ and so there is only 1 choice for $a_{k-2} = a_1$. Similarly, there are $b$ choices for $a_2$ and 1 choice for $a_{k-3}$. By continue this counting, we see that the number of palindromes in this form (when $a$ is already chosen) is equal to $b^{\left\lceil \frac{k-2}{2}\right\rceil}$. Therefore the reciprocal sum of such palindromes satisfies
	\begin{equation*}
		\sum_{a_1,a_2,\ldots,a_{k-2}}\frac{1}{(aa_1a_2\cdots a_{k-2}a)_b}\leq \frac{b^{\left\lceil \frac{k-2}{2}\right\rceil}}{(a00\cdots 0a)_b} \leq \frac{b^{\left\lceil \frac{k-2}{2}\right\rceil}}{a(b^{k-1}+1)},
	\end{equation*}
	where $a_1,a_2,\ldots ,a_{k-2}$ run over all integers $0,1,2,\ldots b-1$ with the usual symmetry conditions of palindromes. Hence
	\begin{equation}\label{thm1eq2}
		s_{b,k} = \sum_{a=1}^{b-1}\sum_{a_1,a_2,\ldots,a_k}\frac{1}{(aa_1a_2\cdots a_{k-2}a)_b}\leq  \sum_{a=1}^{b-1}\frac{b^{\left\lceil \frac{k-2}{2}\right\rceil}}{a(b^{k-1}+1)} = \frac{b^{\left\lceil \frac{k-2}{2}\right\rceil}x}{b^{k-1}+1}.
	\end{equation}
	This implies that 
	\begin{align}\label{eq1}
		s_b &= x+\frac{x}{b+1}+ \sum_{k=3}^{\ell -1} s_{b,k}  +\sum_{k=\ell}^\infty s_{b,k}\nonumber\\
		&\leq \frac{b+2}{b+1}x + \sum_{k=3}^{\ell -1} s_{b,k} + \sum_{k=\ell}^\infty  \frac{b^{\left\lceil \frac{k-2}{2}\right\rceil}x}{b^{k-1}+1}\nonumber\\
		&\leq \frac{b+2}{b+1}x  + \sum_{k=3}^{\ell -1} s_{b,k} + \sum_{k=\ell}^{2m-1}  \frac{b^{\left\lceil \frac{k-2}{2}\right\rceil}x}{b^{k-1}+1}+\sum_{k=2m}^\infty  \frac{b^{\left\lceil \frac{k-2}{2}\right\rceil}x}{b^{k-1}}.
	\end{align}
	If $k = 2t$, then
	\begin{equation*}
		\frac{b^{\left\lceil \frac{k-2}{2}\right\rceil}}{b^{k-1}} = \frac{b^{t-1}}{b^{2t-1}} = \frac{1}{b^t} =  \frac{1}{b^{\left\lfloor  \frac{k}{2}\right\rfloor}}.
	\end{equation*}
	Similarly, if $k = 2t+1$, then 
	\begin{equation*}
		\frac{b^{\left\lceil \frac{k-2}{2}\right\rceil}}{b^{k-1}} = \frac{1}{b^t} =  \frac{1}{b^{\left\lfloor  \frac{k}{2}\right\rfloor}}.
	\end{equation*}
	Therefore
	\begin{equation}\label{eq2}
		\sum_{k=2m}^\infty  \frac{b^{\left\lceil \frac{k-2}{2}\right\rceil}}{b^{k-1}} = \sum_{k=2m}^\infty  \frac{1}{b^{\left\lfloor  \frac{k}{2}\right\rfloor}} = \sum_{k=m}^\infty \frac{2}{b^k} = \frac{2}{(b-1)\cdot b^{m-1}}.
	\end{equation}
	Substituting \eqref{eq2} in \eqref{eq1}, we obtain the desired upper bound for $s_b$. Similarly, if $a\in \{1,2,\ldots b-1\}$ is fixed, then
	\begin{equation*}
		\sum_{a_1, a_2, \ldots, a_{k-2}}\frac{1}{(aa_1a_2\cdots a_{k-2}a)_{b}} \geq \frac{b^{\left\lceil \frac{k-2}{2}\right\rceil}}{(a(b-1)(b-1)\cdots (b-1)a)_{b}} \geq \frac{b^{\left\lceil \frac{k-2}{2}\right\rceil}}{(a+1)b^{k-1}} = \frac{1}{(a+1) b^{\left\lfloor  \frac{k}{2}\right\rfloor}}.
	\end{equation*}
	Therefore 
	\begin{equation}\label{eq3}
		s_{b,k} \geq \sum_{a=1}^{b-1}\frac{1}{(a+1)b^{\left\lfloor  \frac{k}{2}\right\rfloor}} = \frac{y}{b^{\left\lfloor  \frac{k}{2}\right\rfloor}},
	\end{equation}
	where $y=\sum_{a=2}^{b} \frac{1}{a}$. We write
	\begin{equation*}
		s_b = x + \frac{x}{b+1} + \sum_{k=3}^{2\left\lceil \frac{\ell}{2}\right\rceil - 1}s_{b,k} +\sum_{k=2\left\lceil \frac{\ell}{2}\right\rceil}^\infty s_{b,k},
	\end{equation*}
	and apply \eqref{eq3} to the last sum to obtain 
	\begin{align*}
		s_b &\geq x+\frac{x}{b+1} + \sum_{k=3}^{2\left\lceil \frac{\ell}{2}\right\rceil - 1} s_{b,k} +\sum_{k=2\left\lceil \frac{\ell}{2}\right\rceil}^\infty \frac{y}{b^{\left\lfloor  \frac{k}{2}\right\rfloor}}\\
		&= \frac{b+2}{b+1}x+\sum_{k=3}^{2\left\lceil \frac{\ell}{2}\right\rceil - 1} s_{b,k}+2y\sum_{k=\left\lceil \frac{\ell}{2}\right\rceil}^\infty \frac{1}{b^k}\\
		&= \frac{b+2}{b+1}x+\frac{2y}{(b-1)b^{\left\lceil\frac{\ell}{2}\right\rceil - 1}} + \sum_{k=3}^{2\left\lceil \frac{\ell}{2}\right\rceil - 1} s_{b,k}.
	\end{align*}
	This completes the proof.
\end{proof}

The bounds given in Theorem \ref{rpthm1} hold for all $\ell \geq 3$ and $m\geq 2$, so we only need to find a suitable choice of $m$ and $\ell$ to obtain the inequality between $s_b$ for each $b$. Our method works for a wider range of $b$ but we choose to display the bounds of $s_b$ only for $b=2,3,\ldots,16$ since these seem to be the most common bases we are interested in. Nevertheless, we plan to put more data on the approximated values of $s_b$ for $b\leq 10^4$ in the second author's ResearchGate website \cite{Pong2}, which will be freely downloadable by everyone. Now substituting $\ell =5$ and $m=5$ in Theorem \ref{rpthm1}, the bounds of $s_b$ are as follows:
\begin{align*}
	&2.2797059 \leq s_2 \leq 2.5828238 , \quad  2.5870062 \leq s_3 \leq 2.7023567, \quad 2.7730961 \leq s_4 \leq 2.8314382,\\
	&2.9135124 \leq s_5 \leq 2.9481690, \quad  3.0292313 \leq s_6 \leq 3.0520161, \quad  3.1289733 \leq s_7 \leq 3.1450277,\\
	&3.2172728 \leq s_8 \leq 3.2291661, \quad 3.2968399 \leq s_9 \leq 3.3059903, \quad 3.3694527 \leq s_{10} \leq 3.3767037,\\
	&3.4363567 \leq s_{11} \leq 3.4422399, \ \ 3.4984669 \leq s_{12} \leq 3.5033337, \ \ 3.5564805 \leq s_{13} \leq 3.5605718,\\
	&3.6109440 \leq s_{14} \leq 3.6144306, \ \ 3.6622953 \leq s_{15} \leq 3.6653014, \ \  3.7108920 \leq s_{16} \leq 3.7135101.
\end{align*}
From the above inequalities, we easily see that $s_b<s_{b'}$ for $2\leq b< b'\leq 16$. In fact, this inequality holds in general.

\begin{theorem}\label{rpthm2}
	For $2\leq b<b'$, we have $s_b<s_{b'}$.
\end{theorem}

We need the following lemma in the proof of Theorem \ref{rpthm2}.

\begin{lemma}\label{rplemma1}
	Let $a$ and $b$ be integers satisfying $a<b$ and let $f$ be monotone on $[a,b]$. Then
	\begin{equation*}
		\min\{ f(a),f(b) \} \leq \sum_{n=a}^b f(n) - \int_{a}^b f(t) dt \leq \max\{f(a),f(b)\}.
	\end{equation*}
\end{lemma}

\begin{proof}[Proof of Lemma \ref{rplemma1}] 
	This is a well-known result. See, for example, in the book by Nathanson \cite{Nath}. Nevertheless, we give a proof for completeness and for the reader's convenience. We first assume that $f$ is increasing on $[a,b]$. Then for $a\leq n\leq b-1$, we have
	\begin{equation*}
		f(n)\leq \int_n ^{n+1} f(t)dt \leq f(n+1).
	\end{equation*}
	Summing the above from $n=a$ to $n=b-1$, we obtain
	\begin{equation*}
		\sum_{n=a}^{b-1} f(n) \leq \int_a ^b f(t)dt\leq \sum_{n=a+1}^b f(n).
	\end{equation*}
	This implies that 
	\begin{equation*}
		\min\{f(a),f(b)\} = f(a) \leq \sum_{n=a}^b f(n) - \int_a ^b f(t)dt \leq f(b) = \max\{f(a),f(b)\}.
	\end{equation*}
	The proof is similar when $f$ is decreasing.
\end{proof}

\begin{proof}[Proof of Theorem \ref{rpthm2}] 
	We first apply Theorem \ref{rpthm1} with $\ell =5$ and $m=5$ and run the computation in a computer to obtain $s_b < s_{b'}$ for all $2\leq b < b' \leq 50$. Therefore it suffices to show that $s_b<s_{b+1}$ for $b\geq 50$. So we assume throughout the proof that $b\geq 50$. From \eqref{thm1eq1} in the proof of Theorem \ref{rpthm1}, we have 
	\begin{equation*}
		\sum_{k=1}^2 s_{b,k}  = x_b + \frac{x_b}{b+1} = \frac{b+2}{b+1}x_b.
	\end{equation*}
	Similarly, $\sum_{k=1}^2 s_{b+1,k} = \frac{b+3}{b+2}x_{b+1}$. Therefore
	\begin{align}\label{eq4}
		\sum_{k=1}^2 (s_{b+1,k} - s_{b,k}) &= \frac{b+3}{b+2}x_{b+1} - \frac{b+2}{b+1}x_b\nonumber\\
		&= \frac{b+3}{b+2}\left(x_b + \frac{1}{b}\right) - \frac{b+2}{b+1}x_b\nonumber\\
		&= \frac{1}{b} + \frac{1}{b(b+2)} - \frac{x_b}{(b+1)(b+2)}\nonumber\\
		&> \frac{1}{b} - \frac{x_b}{(b+1)(b+2)} > \frac{1}{b} - \frac{x_b}{b(b-1)}.
	\end{align}
	Similar to \eqref{thm1eq2} and \eqref{eq2} in the proof of Theorem \ref{rpthm1}, we obtain
	\begin{equation}\label{eq5}
		\sum_{k=4}^\infty s_{b,k} \leq \sum_{k=4}^\infty \frac{b^{\left\lceil\frac{k-2}{2}\right\rceil}x_b}{b^{k-1}+1} \leq \sum_{k=4}^\infty \frac{b^{\left\lceil\frac{k-2}{2}\right\rceil}x_b}{b^{k-1}} = \sum_{k=2}^\infty \frac{2x_b}{b^k} = \frac{2x_b}{b(b-1)}.
	\end{equation}
	By Lemma \ref{rplemma1}, we have 
	\begin{equation*}
		x_b = \sum_{n=1}^{b-1} \frac{1}{n} \leq 1 + \int_{1}^{b-1} \frac{1}{t} dt < 1 + \log b <2\log b.
	\end{equation*}
	Therefore we obtain from \eqref{eq4} and \eqref{eq5} that
	\begin{equation}\label{eq6}
		\sum_{k=1}^2 (s_{b+1,k} - s_{b,k}) - \sum_{k=4}^\infty s_{b,k} > \frac{1}{b} - \frac{3x_b}{b(b-1)} > \frac{1}{b} - \frac{6\log b}{b(b-1)}.
	\end{equation}
	Next, we will show that 
	\begin{equation}\label{thm2eqprove}
		\left(\sum_{k=3}^{\infty} s_{b+1,k} \right) - s_{b,3} \geq -\frac{5}{b^2}.
	\end{equation}
	Before we prove \eqref{thm2eqprove}, we need to verify the following inequalities. For $1\leq a \leq \left\lfloor\frac{b}{2}\right\rfloor -1$ and $0\leq c \leq b-(2a+2)$,
	\begin{equation}\label{thm2eq1}
		\frac{1}{(aca)_{b+1}} - \frac{1}{(a(2a+c+1)a)_b} > 0.
	\end{equation}
	For $1\leq a \leq \left\lfloor\frac{b}{2}\right\rfloor -1$ and $b-(2a+1)\leq c \leq b$,
	\begin{equation}\label{thm2eq2}
		\frac{1}{(aca)}_{b+1} - \frac{1}{((a+1)(2a+c+2-b)(a+1))_b} > 0.
	\end{equation}
	For $\left\lfloor\frac{b}{2}\right\rfloor \leq a \leq b-3$ and $0\leq c \leq b-1$,
	\begin{equation}\label{thm2eq3}
		\frac{1}{(aca)_{b+1}} - \frac{1}{((a+2)c(a+2))_b} > 0.
	\end{equation}
	Note that in \eqref{thm2eq2}, if $b$ is even, $a=\frac{b}{2}-1$, and $c=b$, then $2a+c+2-b=b$ and the expression
	\begin{equation*}
		((a+1)(2a+c+2-b)(a+1))_b = ((a+1)b(a+1))_b
	\end{equation*}
	is not a standard representation of the number in base $b$. In this case, for convenience, we still define $(a_k a_{k-1} \ldots a_1 a_0)_b = a_k b^k + a_{k-1}b^{k-1} + \cdots a_0$ where $a_i$'s are not necessarily less than $b$. For \eqref{thm2eq1}, the difference between the denominators in the first and second fractions is 
	\begin{equation*}
		a(b+1)^2 +c(b+1)+a - (ab^2 + (2a+c+1)b +a) = -b + a + c \leq - a -2 < 0,
	\end{equation*}
	where the first inequality is obtained from the assumption that $c \leq b - (2a+2)$. Therefore
	\begin{equation*}
		\frac{1}{(aca)_{b+1}} = \frac{1}{a(b+1)^2 +c(b+1)+a} > \frac{1}{ab^2 + (2a+c+1)b +a} = \frac{1}{(a(2a+c+1)a)_b}.
	\end{equation*}
	This proves \eqref{thm2eq1}. Similarly, the difference between the first and second denominators in \eqref{thm2eq2} is 
	\begin{align*}
		-2b + a + c - 1 \leq - b + a - 1 \leq  -b + \left\lfloor\frac{b}{2}\right\rfloor - 2 < 0,
	\end{align*}
	which implies \eqref{thm2eq2}. For \eqref{thm2eq3}, the difference between the first and second denominators in \eqref{thm2eq3} is
	\begin{align*}
		-2b^2 + 2ab + a + c - 2 \leq -2b^2 + 2(b-3)b + (b-3) + (b-1) - 2 =  -4b - 6 < 0,
	\end{align*}
	which implies \eqref{thm2eq3}. Now we are ready to prove \eqref{thm2eqprove}. Similar to \eqref{eq3} in the proof of Theorem \ref{rpthm1}, we have
	\begin{equation}\label{thm2eq4}
		\sum_{k=4}^\infty s_{b+1,k} \geq \sum_{k=4}^\infty \frac{y_{b+1}}{(b+1)^{\left\lfloor\frac{k}{2}\right\rfloor}} = 2y_{b+1}\sum_{k=2}^\infty \frac{1}{(b+1)^k} \geq \frac{2y_b}{b(b+1)}.
	\end{equation} 
	Next, we divide the proof into two cases depending on the parity of $b$.
	
	\noindent{\bf Case 1.} $b$ is odd. We write $s_{b+1,3}$ as
	\begin{align*}
		s_{b+1,3} &= \sum_{a=1}^b \sum_{c=0}^b \frac{1}{(aca)_{b+1}} \\
		&= \sum_{a=1}^{\frac{b-3}{2}}\sum_{c=0}^{b-2a-2} \frac{1}{(aca)_{b+1}} + \sum_{a=1}^{\frac{b-3}{2}}\sum_{c=b-2a-1}^{b} \frac{1}{(aca)_{b+1}} + \sum_{a=\frac{b-1}{2}}^{b-3}\sum_{c=0}^{b-1} \frac{1}{(aca)_{b+1}} + \sum_{a=\frac{b-1}{2}}^{b-3} \frac{1}{(aba)_{b+1}}\\
		&\phantom{=} +\sum_{a=b-2}^{b}\sum_{c=0}^b \frac{1}{(aca)_{b+1}}\\
		&= A_1 + A_2 + A_3 + A_4 + A_5, \text{ say}.
	\end{align*}
	In addition, we write $s_{b,3}$ as
	\begin{align*}
	s_{b,3} &= \sum_{a=1}^{b-1} \sum_{c=0}^{b-1} \frac{1}{(aca)_b}\\
	&=\sum_{c=0}^2 \frac{1}{(1c1)_b} + \sum_{a=2}^{\frac{b-1}{2}} \frac{1}{(a0a)_b} + \sum_{a=2}^{\frac{b-1}{2}}\sum_{c=1}^{2a} \frac{1}{(aca)_b} + \sum_{a=1}^{\frac{b-1}{2}}\sum_{c=2a+1}^{b-1} \frac{1}{(aca)_b} + \sum_{c=0}^{b-1} \frac{1}{\left( \left(\frac{b+1}{2}\right)c\left(\frac{b+1}{2}\right) \right)_b}\\ 
	&\phantom{=} + \sum_{a=\frac{b+3}{2}}^{b-1}\sum_{c=0}^{b-1} \frac{1}{(aca)_{b}}\\ 
	&=B_1 + B_2 + B_3 + B_4 + B_5 + B_6, \text{ say}.
	\end{align*}
	We write the difference $s_{b+1,3} - s_{b,3}$ as 
	\begin{equation}\label{thm2eq5}
		(A_1 -  B_4) + (A_2 - B_3) + (A_3  - B_6) + A_4 + A_5 - B_1 - B_2 - B_5. 
	\end{equation}
 	Observe that the following inequalities holds.
 	\begin{align}
 		A_4 &> 0, \quad A_5 >0,\nonumber\\
 		B_1 &= \sum_{c=0}^2 \frac{1}{(1c1)_b} = \frac{1}{(101)_b} + \frac{1}{(111)_b} + \frac{1}{(121)_b} < \frac{3}{(100)_b} = \frac{3}{b^2},\label{thm2eq6}\\
 		B_2 &= \sum_{a=2}^{\frac{b-1}{2}} \frac{1}{(a0a)_b} = \frac{1}{b^2 +1}\sum_{a=2}^{\frac{b-1}{2}} \frac{1}{a}  < \frac{1}{b^2 + 1}\sum_{a=2}^{b} \frac{1}{a} = \frac{y_b}{b^2 +1},\label{thm2eq7}\\
 		B_5 &= \sum_{c=0}^{b-1} \frac{1}{\left( \left(\frac{b+1}{2}\right)c\left(\frac{b+1}{2}\right) \right)_b} < \sum_{c=0}^{b-1} \frac{1}{\left( \left(\frac{b+1}{2}\right)00 \right)_b} = \frac{b}{\left( \left(\frac{b+1}{2}\right)00 \right)_b} < \frac{2}{b^2}.\label{thm2eq8}
 	\end{align}
 	Therefore \eqref{thm2eq5} is greater than
 	\begin{equation}\label{thm2eq9}
 		(A_1 -  B_4) + (A_2 - B_3) + (A_3  - B_6) - \frac{5}{b^2} - \frac{y_b}{b^2 + 1}.
 	\end{equation}
 	When $a=\frac{b-1}{2}$, the sum $\sum_{c=2a+1}^{b-1} \frac{1}{(aca)_b}$ in $B_4$ is empty. From this and the change of variable from $c$ to $c+2a+1$,  we obtain
 	\begin{equation*}
 		B_4 = \sum_{a=1}^{\frac{b-3}{2}}\sum_{c=2a+1}^{b-1} \frac{1}{(aca)_b} = \sum_{a=1}^{\frac{b-3}{2}} \sum_{c=0}^{b-2a-2} \frac{1}{(a(2a+c+1)a)_b}.
 	\end{equation*}
 	We see that 
 	\begin{equation*}
 		A_1 - B_4 = \sum_{a=1}^{\frac{b-3}{2}} \sum_{c=0}^{b-2a-2} \left( \frac{1}{(aca)_{b+1}} - \frac{1}{(a(2a+c+1)a)_b} \right).
 	\end{equation*}
 	Similarly, by replacing the variables $a$ by $a+1$ and $c$ by $2a+c+2-b$ in $B_3$, we see that
 	\begin{equation*}
 		A_2 - B_3 = \sum_{a=1}^{\frac{b-3}{2}}\sum_{c=b-2a-1}^b \left( \frac{1}{(aca)}_{b+1} - \frac{1}{((a+1)(2a+c+2-b)(a+1))_b} \right),
 	\end{equation*}
 	and by changing $a$ to $a+2$ in $B_6$, we obtain
 	\begin{equation*}
 		A_3 - B_6 \sum_{a=\frac{b-1}{2}}^{b-3}\sum_{c=0}^{b-1} \left( \frac{1}{(aca)_{b+1}} - \frac{1}{((a+2)c(a+2))_b} \right).
 	\end{equation*}
 	By \eqref{thm2eq1}, \eqref{thm2eq2}, and \eqref{thm2eq3}, respectively, we have
 	\begin{equation}\label{thm2eq10}
 		A_1 - B_4 > 0, \quad A_2 - B_3 >0, \quad\text{and}\quad A_3 - B_6>0.
 	\end{equation}
 	Then we obtain by \eqref{thm2eq5}, \eqref{thm2eq9}, and \eqref{thm2eq10} that $s_{b+1,3} - s_{b,3}$ is greater than $-\frac{5}{b^2} - \frac{y_b}{b^2 +1}$. From this and \eqref{thm2eq4}, we have
 	\begin{align*}
	 	\sum_{k=3}^{\infty} s_{b+1,k} - s_{b,3} = s_{b+1,3} - s_{b,3} + \sum_{k=4}^{\infty} s_{b+1,k}  &> -\frac{5}{b^2} - \frac{y_b}{b^2 +1} + \frac{2y_b}{b(b+1)}\\
	 	&= -\frac{5}{b^2}  + \frac{b^2 - b +2}{b(b+1)(b^2+1)}y_b,
 	\end{align*}
	which is larger than $-\frac{5}{b^2}$. Therefore \eqref{thm2eqprove} is proved as desired.
	
	\noindent{\bf Case 2.} $b$ is even. The argument in this case is similar to Case 1, so we omit some details. We express $s_{b+1,3}$ as
	\begin{align*}
		s_{b+1,3} &= \sum_{a=1}^{\frac{b-2}{2}}\sum_{c=0}^{b-2a-2} \frac{1}{(aca)_{b+1}} + \sum_{a=1}^{\frac{b-2}{2}}\sum_{c=b-2a-1}^{b} \frac{1}{(aca)_{b+1}} + \sum_{a=\frac{b}{2}}^{b-3}\sum_{c=0}^{b-1} \frac{1}{(aca)_{b+1}} + \sum_{a=\frac{b}{2}}^{b-3} \frac{1}{(aba)_{b+1}}\\
		&\phantom{=} +\sum_{a=b-2}^{b}\sum_{c=0}^b \frac{1}{(aca)_{b+1}}\\
		&> \sum_{a=1}^{\frac{b-2}{2}}\sum_{c=0}^{b-2a-2} \frac{1}{(aca)_{b+1}} + \sum_{a=1}^{\frac{b-2}{2}}\sum_{c=b-2a-1}^{b} \frac{1}{(aca)_{b+1}} + \sum_{a=\frac{b}{2}}^{b-3}\sum_{c=0}^{b-1} \frac{1}{(aca)_{b+1}}\\
		&= A_1 + A_2 + A_3, \text{ say}.
	\end{align*}
	In writing $s_{b,3}$, we still use our convention that $(a_3 a_2 a_1)_b = a_3 b^2 + a_2 b + a_1$ for $0\leq a_3 , a_2 , a_1 \leq b$. Then we have
	\begin{align*}
		s_{b,3} &=\sum_{c=0}^2 \frac{1}{(1c1)_b} + \sum_{a=2}^{\frac{b}{2}} \frac{1}{(a0a)_b} + \left(\left(\sum_{a=2}^{\frac{b}{2}}\sum_{c=1}^{2a} \frac{1}{(aca)_b}\right) - \frac{1}{\left(\left(\frac{b}{2}\right)b\left(\frac{b}{2}\right)\right)_b}\right) + \sum_{a=1}^{\frac{b}{2}}\sum_{c=2a+1}^{b-1} \frac{1}{(aca)_b}\\
		&\phantom{=} + \sum_{c=0}^{b-1} \frac{1}{\left( \left(\frac{b+2}{2}\right)c\left(\frac{b+2}{2}\right) \right)_b}  + \sum_{a=\frac{b+4}{2}}^{b-1}\sum_{c=0}^{b-1} \frac{1}{(aca)_{b}}.
	\end{align*}
	Similar to \eqref{thm2eq6}, \eqref{thm2eq7}, and \eqref{thm2eq8}, $s_{b,3}$ is less than
	\begin{align*}
		&\frac{5}{b^2} + \frac{y_b}{b^2 + 1} + \sum_{a=2}^{\frac{b}{2}}\sum_{c=1}^{2a} \frac{1}{(aca)_b} + \sum_{a=1}^{\frac{b}{2}}\sum_{c=2a+1}^{b-1} \frac{1}{(aca)_b} + \sum_{a=\frac{b+4}{2}}^{b-1}\sum_{c=0}^{b-1} \frac{1}{(aca)_{b}}\\
		&=\frac{5}{b^2} + \frac{y_b}{b^2 + 1} + B_1 + B_2 + B_3, \text{ say}.
	\end{align*}
	When $a=\frac{b}{2}$, the sum $\sum_{c=2a+1}^{b-1} \frac{1}{(aca)_b}$ in $B_2$ is empty. So
	\begin{equation*}
		B_2 = \sum_{a=1}^{\frac{b-2}{2}}\sum_{c=2a+1}^{b-1} \frac{1}{(aca)_b} = \sum_{a=1}^{\frac{b-2}{2}} \sum_{c=0}^{b-2a-2} \frac{1}{(a(2a+c+1)a)_b}.
	\end{equation*}
	Then the difference $s_{b+1,3} - s_{b,3}$ is larger than
	\begin{equation*}
		(A_1 - B_2) + (A_2 - B_1) + (A_3 - B_3) - \frac{5}{b^2} - \frac{y_b}{b^2 +1}
	\end{equation*}
	Similar to case 1, by \eqref{thm2eq1}, \eqref{thm2eq2}, and \eqref{thm2eq3}, respectively, we have
	\begin{equation*}
		A_1 - B_2 > 0, \quad A_2 - B_1 > 0, \quad \text{and} \quad A_3 - B_3 >0.
	\end{equation*}
	Thus 
	\begin{equation*}
		s_{b+1,3} - s_{b,3} > - \frac{5}{b^2} - \frac{y_b}{b^2 +1}.
	\end{equation*}
	Similar to Case 1, we therefore obtain by \eqref{thm2eq4} that
	\begin{equation*}
		\sum_{k=3}^{\infty} s_{b+1,k} - s_{b,3} = s_{b+1,3} - s_{b,3} + \sum_{k=4}^{\infty} s_{b+1,k}  > -\frac{5}{b^2} - \frac{y_b}{b^2 +1} + \frac{2y_b}{b(b+1)} > -\frac{5}{b^2}.
	\end{equation*}
 	Hence, in any cases, \eqref{thm2eqprove} is verified. From \eqref{eq6} and \eqref{thm2eqprove}, we obtain
	\begin{equation}\label{thm2eq11}
		s_{b+1} - s_b = \left(\sum_{k=1}^2 (s_{b+1,k} - s_{b,k}) - \sum_{k=4}^\infty s_{b,k}\right) + \left(\sum_{k=3}^{\infty} s_{b+1,k} - s_{b,3}\right) > \frac{1}{b} - \frac{6\log b}{b(b-1)} - \frac{5}{b^2}.
	\end{equation}
	Observe that 
	\begin{equation*}
		\frac{d}{dx}\left( \frac{6\log x}{x-1} \right) = -\frac{6(x\log x-x+1)}{x(x-1)^2},
	\end{equation*}
	and $-6(x\log x-x+1)< 0$ for all $x\geq 50$. So the function $x\mapsto \frac{6\log x}{x-1}$ is decreasing on $[50,\infty)$. Since $b\geq 50$,
	\begin{equation*}
		\frac{6\log b}{b-1} \leq \frac{6\log 50}{49} <\frac{1}{2} \quad \text{and} \quad \frac{5}{b} <\frac{1}{3}.
	\end{equation*}
	Hence we obtain from \eqref{thm2eq11} that
	\begin{equation*}
		s_{b-1} - s_b > \frac{1}{b} - \frac{6\log b}{b(b-1)} - \frac{5}{b^2} > \frac{1}{b} - \frac{1}{2b} - \frac{1}{3b} > 0.
	\end{equation*}
	This  completes the proof.
\end{proof}

Next we give some consequences of our main theorems.

\begin{corollary}\label{rpcor2}
	Let $L_{b}$ and $U_b$ be, respectively, the lower and upper bounds for $s_b$ given in Corollary \ref{rpcor1}. Then $U_b - L_b$ converges to zero as $b\rightarrow \infty$. In addition, uniformly for $b\geq 2$,
	\begin{equation*}
		s_b = \frac{b+2}{b+1}\left( \log b + \gamma \right) + O\left(\frac{\log b}{b}\right),
	\end{equation*}
	where $\gamma$ is Euler's constant.
\end{corollary}

\begin{proof}
	We have 
	\begin{equation*}
		U_b - L_b = \frac{2(x_b - y_b)}{b(b-1)} + \frac{bx_b}{b^2 + 1} = \frac{2}{b^2} + \frac{bx_b}{b^2 +1}.
	\end{equation*}
	By Lemma \ref{rplemma1}, we see that $x_b \leq 1 + \log b$. Therefore 
	\begin{equation*}
		0\leq U_b - L_b \leq \frac{2}{b^2} + \frac{1 + \log b}{b} \ll \frac{\log b}{b} \rightarrow 0 \quad \text{as $b\rightarrow \infty$}.
	\end{equation*}
	This proves the first assertion. In addition, we have 
	\begin{equation*}
		0 \leq s_b - L_b \leq U_b - L_b \ll \frac{\log b}{b}.
	\end{equation*}
	So 
	\begin{equation}\label{cor2eq1}
	s_b = L_b + O\left(\frac{\log b}{b}\right).
	\end{equation}
	Recall a well-known formula (see for instance in \cite{Nath}),
	\begin{equation*}
		\sum_{n\leq x} \frac{1}{n} = \log x + \gamma + O\left( \frac{1}{x} \right).
	\end{equation*}
	Therefore
	\begin{equation*}
		L_b = \frac{b+2}{b+1}\left( \log (b-1) + \gamma + O\left( \frac{1}{b-1} \right) \right) + \frac{2}{b(b-1)}\left( \log b + \gamma - 1 + O\left( \frac{1}{b} \right) \right).
	\end{equation*}
	Since $\log (b-1) = \log b + O\left(\frac{1}{b}\right)$, we see that
	\begin{equation}\label{cor2eq2}
	L_b = \frac{b+2}{b+1}\left( \log b + \gamma\right) + O\left(\frac{1}{b}\right).
	\end{equation}
	From \eqref{cor2eq1} and \eqref{cor2eq2}, we obtain the second assertion.
\end{proof}

\begin{corollary}\label{rpcor3}
	The sequence $(s_b)_{b\geq 2}$ is strictly increasing and diverges to $+\infty$ as $b \rightarrow \infty$.
\end{corollary}

\begin{proof}
	This follows immediately from Theorem \ref{rpthm2} and Corollary \ref{rpcor2}.
\end{proof}

\begin{corollary}\label{rpcor4}
	The sequence $(s_b - s_{b-1})_{b\geq 3}$ converges to zero as $b \rightarrow \infty$.
\end{corollary}

\begin{proof}
	Recall that $\log (b-1) = \log b + O\left(\frac{1}{b}\right)$. So we obtain by Theorem \ref{rpthm2} and \ref{rpcor2} that, for $b\geq 3$,
	\begin{equation*}
		0 < s_b - s_{b-1} = \left( \frac{b+2}{b+1} - \frac{b+1}{b} \right)\left( \log b + \gamma \right) + O\left( \frac{\log (b-1)}{b-1} \right) \ll \frac{\log (b-1)}{b-1},
	\end{equation*}
	which implies our assertion.
\end{proof}

\section{Numerical Data}
In this section, we give some comments on the values of $s_b ^2 - s_{b-1}s_{b+1}$ for each $b\geq 2$. Recall that a sequence $(a_n)_{n\geq 0}$ is said to be log--concave if $a_n ^2 - a_{n-1}a_{n+1} > 0$ for every $n\geq 1$ and is said to be log--convex if $a_n ^2 - a_{n-1}a_{n+1} < 0$ for every $n\geq 1$. For a survey article concerning log--concavity and log--convexity of sequences, we refer the reader to Stanley \cite{Stan}. See also Pongsriiam \cite{Pong} for some combinatorial sequences which are log--concave or log--convex, and some open problems concerning log--property of a certain sequence.

\begin{table}[h!]
	\centering
	\caption{The approximated values of $L(b)$, $M(b)$, and $U(b)$.}
	\begin{tabular}{|c|c|c|c|}
		\hline
		b                        & L(b)                  & M(b)                  & U(b)        \\ \hline
		\ 3 \       &\ \  --0.62050401 \ \     &\ \ 0.18128669  \ \    &\ \ 0.98088799  \ \  \\
		 4       &\ \  --0.27694197 \ \      &\ \ 0.10156088 \ \     &\ \ 0.47976680 \ \  \\
		 5       &\ \  --0.15303980 \ \     &\ \ 0.06918746 \ \     &\ \ 0.29135060 \ \  \\
		 6       &\ \  --0.09583068 \ \     &\ \ 0.05134357  \ \    &\ \ 0.19849920 \ \  \\
		 7       &\ \ --0.06499252 \ \     &\ \ 0.04017479 \ \     &\ \ 0.14533547 \ \  \\
		 8     &\ \ --0.04658631 \ \     &\ \ 0.03260281 \ \     &\ \ 0.11178921 \ \  \\
		 9      &\ \ --0.03478306 \ \     &\ \ 0.02717043 \ \     &\ \ 0.08912266 \ \  \\
		10     &\ \ --0.02679941 \ \    &\ \  0.02310516 \ \    &\ \  0.07300909 \ \  \\
		11      &\ \ --0.02117156 \ \    &\ \  0.01996245 \ \    &\ \  0.06109613 \ \  \\
		12      &\ \ --0.01707119 \ \    &\ \  0.01746963 \ \    &\ \  0.05201026 \ \  \\
		13      &\ \ --0.01400177 \ \    &\ \  0.01545069 \ \    &\ \  0.04490303 \ \  \\
		14      &\ \ --0.01165154 \ \    &\ \  0.01378717 \ \    &\ \  0.03922582 \ \  \\
		15      &\ \ --0.00981708 \ \    &\ \  0.01239658 \ \    &\ \  0.03461020 \ \  \\
		16      &\ \ --0.00836134 \ \    &\ \  0.01121975 \ \    &\ \  0.03080080 \ \  \\
		17      &\ \ --0.00718938 \ \    &\ \  0.01021318 \ \    &\ \  0.02761573 \ \   \\
		18      &\ \ --0.00623386  \ \   &\ \  0.00934426 \ \    &\ \  0.02492236 \ \  \\
		19      &\ \ --0.00544602 \ \    &\ \  0.00858800 \ \    &\ \  0.02262202 \ \  \\
		20      &\ \ --0.00478989 \ \    &\ \  0.00792504 \ \    &\ \  0.02063996 \ \  \\  \hline       
	\end{tabular}
\end{table}

Since we do not know the exact value of $s_b$, it is difficult to determine if the sequence $(s_b)_{b\geq 2}$ is log--concave or log--convex, or neither. Nevertheless, we can use Theorem \ref{rpthm1} to estimate $s_b ^2 -s_{b-1}s_{b+1}$ and use numerical data to predict the log--concavity of the sequence $(s_b)_{b\geq 2}$. For each $b\geq 2$, let $\alpha_b$ and $\beta_b$ be the upper and lower bounds of $b$ given in Theorem \ref{rpthm1} with $\ell = 5$ and $m = 5$, respectively. In addition, for each $b\geq 3$, we define
\begin{align*}
	L(b) &= \beta_b ^2 - \alpha_{b-1}\alpha_{b+1}, \\
	M(b) &= \left(\frac{\alpha_b + \beta_b}{2}\right) ^2 - \left(\frac{\alpha_{b-1} + \beta_{b-1}}{2}\right)\left(\frac{\alpha_{b+1} + \beta_{b+1}}{2}\right),\\
	U(b) &= \alpha_b ^2 - \beta_{b-1}\beta_{b+1}.
\end{align*}
Therefore $L(b)\leq s_b ^2 - s_{b-1}s_{b+1} \leq U(b)$ and we expect that $M(b)$ should be a good approximation for $s_b ^2 -s_{b-1}s_{b+1}$. The table of approximated values of $L(b)$, $M(b)$, and $U(b)$ is given above.

We see that for each $3 \leq b \leq 20$, $M(b)$ is positive. In fact, we check by using MATLAB that $M(b)>0$ for $2\leq b\leq 500$. So if $s_b ^2 -s_{b-1}s_{b+1}$ is very closed to $M(b)$, then we guess that it should also be positive. So we think that the sequence $(s_b)_{b\geq 2}$ is log--concave. Nevertheless, we do not have a proof of this. So we leave this for a future research and we do not mind if the reader will solve it. We also plan to put more data in the second author's ResearchGate website \cite{Pong2}, which everyone can visit and freely  download the data.

\begin{conjecture}
	The sequence $(s_b)_{b\geq 2}$ is log--concave.
\end{conjecture}

\section{Acknowledgments}

Phakhinkon Phunphayap receives a scholarship from Science Achievement Scholarship of Thailand(SAST). Prapanpong Pongsriiam receives financial support jointly from The Thailand Research Fund and Faculty of Science Silpakorn University, grant number RSA5980040.


\begin{thebibliography}{99}
\bibitem{Adam1} B. Adamczewski and J.-P. Allouche, Reversal and palindromes in continued fractions, Theoret. Comput. Sci. 320 (2007), 220--237.
\bibitem{Adam2} B. Adamczewski and Y. Bugeaud, On the Littlewood conjecture in simultaneous Diophantine approximation, J. Lond. Math. Soc. 73 (2006), 355--366.
\bibitem{Adam3} B. Adamczewski and Y. Bugeaud, Palindromic continued fractions, Ann. Inst. Fourier 57 (2007), 1557--1574.
\bibitem{Adam4} B. Adamczewski and Y. Bugeaud, Transcendence measure for continued fractions involving repetitive or symmetric patterns, J. Eur. Math. Soc. 12 (2010),
883--914.
\bibitem{Amb} P. Ambro\v{z}, C. Frougny, Z. Mas\'{a}kov\'{a}, and E. Pelantov\'{a}, Palindromic complexity of infinite words associated with simple Parry numbers, Ann. Inst. Fourier
(Grenoble) 56 (2006), 2131--2160.

\bibitem{BSh} R. G. Bartle and D. R. Sherbert, Introduction to Real Analysis, John Wiley$\&$Sons, 1992.
\bibitem{Bank} W. D. Banks, Every natural number is the sum of forty-nine palindromes, Integers 16 (2016), A3, 9 pp. 
\bibitem{Bank2} W. D. Banks, D. N. Hart, and M. Sakata, Almost all palindromes are composite, Math. Res. Lett. 11 (2004), no. 5-6, 853-868.
\bibitem{Bank3}  W. D. Banks and I. Shparlinski, Prime divisors of palindromes, Period. Math. Hungar. 51 (2005), 1-10.
\bibitem{Bay} J. Bayless and D. Klyve, Reciprocal sums as a knowledge metric: theory, computation, and perfect numbers, \textit{Amer. Math. Monthly} 120 (2013), no. 9, 822-831.

\bibitem{Brl} S. Brlek, S. Hamel, M. Nivat, and C. Reutenauer, On The Palindromic Complexity Of Infinite Words, Internat. J. Found. Comput. Sci. 15 (2004), 293--306.
\bibitem{Bug} Y. Bugeaud, and M. Laurent, Exponents of Diophantine and Sturmian continued fractions, Ann. Inst. Fourier 55 (2005), 773--804.

\bibitem{Cil} J. Cilleruelo, F. Luca, and L. Baxter, Every positive integer is a sum of three palindromes, \textit{Math. Comp.} electronically published on August 15, 2017, DOI: https://doi.org/10.1090/mcom/3221 (to appear in print).

\bibitem{Dam} D. Damanik and L. Q. Zamboni, Combinatorial properties of Arnoux–Rauzy subshifts and applications to Schr\"{o}dinger operators, Rev. Math. Phys. 15 (2003),
745--763.
\bibitem{Fic} G. Ficia and L. Q. Zamboni, On the least number of palindromes contained in an infinite word, Theoret. Comput. Sci. 481 (2013), 1--8.
\bibitem{Fis} S. Fischler, Palindromic prefixes and diophantine approximation, Monatsh. Math. 151 (2007), 11--37.

\bibitem{Goin} E. Goins, Palindromes in different bases: A conjecture of J. Ernest Wilkins, Integers 9 (2009), A55, 725-734.
\bibitem{Grim} R. P. Grimaldi, Compositions and the alternate Fibonacci numbers, Congr. Numer. 186 (2007), 81-96.

\bibitem{Hof} A. Hof, O. Knill, and B. Simon, Singular continuous spectrum for palindromic Schr\"{o}dinger operators, Comm. Math. Phys. 174 (1995), 149--159.

\bibitem{Luca} F. Luca, Palindromes in various sequences, Gac. R. Soc. Mat. Esp. 20 (2017), no. 1, 49-56. 
\bibitem{Nath} M. B. Nathanson, Elementary Methods in Number Theory, Springer, 1999.
\bibitem{Pong} P. Pongsriiam, Local Behaviors of the Number of Relatively Prime Sets, Int. J. Number Theory 12 (2016), 1575--1593.
\bibitem{Pong2} P. Pongsriiam's ResearchGate account available at https://www.researchgate.net/profile/Prapanpong \_Pongsriiam
\bibitem{Ra} A. Rajasekaran, J. Shallit, and T. Smith, Sums of Palindromes: an Approach via Automata, preprint available at arXiv:1706.10206.

\bibitem{Rog} P. Roggero, M. Nardelli, and F. Di Noto, Sum of the reciprocals of famous series: mathematical connections with some sectors of theoretical physics and string theory, preprint available at http://empslocal.ex.ac.uk/people/staff/mrwatkin/zeta/nardelli2017a.pdf
\bibitem{Roy} D. Roy, Approximation to real numbers by cubic algebraic integers, II, Ann. of Math. 158 (2003), 1081--1087.
\bibitem{Roy2} D. Roy, Approximation to real numbers by cubic algebraic integers, I, Proc. Lond. Math. Soc. 88 (2004), 42--62.

\bibitem{Slo} N. J. A. Sloane, The On-Line Encyclopedia of Integer Sequences, https://oeis.org.
\bibitem{Stan} R. P. Stanley, Log-concave and unimodal sequences in algebra, combina-torics, and geometry. Ann. New York Acad. Sci. 576(1989), 500–535.
\bibitem{Math} Mathematics Stack Exchange,  https://math.stackexchange.com/questions/2432424/convergence-of-sum-of-reciprocal-palindromes/2432432



\end{thebibliography}
\end{document}